\documentclass[12pt]{amsart}
\usepackage{tikz}

\usepackage{amsmath,amsfonts,amsthm,amssymb,amscd, verbatim, graphicx,textcomp, hyperref}
\makeatletter
\newcommand{\xRightarrow}[2][]{\ext@arrow 0359\Rightarrowfill@{#1}{#2}}
\makeatother
\usepackage[notcite, notref]{}
\usepackage{upgreek}
\usepackage[margin=2.5cm]{geometry}
\usepackage{pinlabel}
\usepackage{pdflscape}
\usepackage{tabularx}
\usepackage{latexsym}
\usepackage{hyperref}
\usepackage{multirow}
\usepackage{xypic}
\usepackage{enumitem}
\usepackage{microtype}

\usepackage{color}
\usepackage{stmaryrd}

\usepackage{makecell}

\newtheorem{Thm}{Theorem}[section]

\newtheorem{Prop}[Thm]{Proposition}

\newtheorem{Lem}[Thm]{Lemma}
\newtheorem{Conj}[Thm]{Conjecture}

\theoremstyle{definition}

\newtheorem{Rem}[Thm]{Remark}

\theoremstyle{remark}

\numberwithin{equation}{section}

\newcommand{\Aut}{\operatorname{Aut}}

\newcommand{\Hom}{\operatorname{Hom}}

\newcommand{\st}{\operatorname{st}}

\newcommand{\Out}{\operatorname{Out}}

\newcommand{\Inn}{\operatorname{Inn}}

\newcommand{\Res}{\operatorname{Res}}

\newcommand{\cl}{\operatorname{cl}}

\def \a {\alpha}\def \b {\beta}

\newcommand{\m}{\operatorname{\mathbf{m}}}
\renewcommand{\c}{\textbf{c}}

\renewcommand{\t}{\textbf{t}}

\renewcommand{\a}{\textbf{a}}
\newcommand{\w}{\operatorname{\mathbf{w}}}
\renewcommand{\b}{\textbf{b}}
\newcommand{\z}{\textbf{z}}

\newcommand{\x}{\textbf{x}}
\newcommand{\res}{\operatorname{res}}

\newcommand{\Irr}{\operatorname{Irr}}

\newcommand{\MM}{\operatorname{M}}

\newcommand{\J}{\operatorname{J}}
\newcommand{\PSL}{\operatorname{PSL}}
\newcommand{\PGL}{\operatorname{PGL}}
\newcommand{\PSU}{\operatorname{PSU}}
\newcommand{\Th}{\operatorname{Th}}
\newcommand{\He}{\operatorname{He}}
\newcommand{\Fi}{\operatorname{Fi}}
\newcommand{\ON}{\operatorname{O'N}}
\newcommand{\GL}{\operatorname{GL}}
\newcommand{\SL}{\operatorname{SL}}

\newcommand{\Ind}{\operatorname{Ind}}
 
\renewcommand{\epsilon}{\varepsilon}
\renewcommand{\bar}{\overline}
\renewcommand{\hat}{\widehat}
\renewcommand{\leq}{\leqslant}
\renewcommand{\geq}{\geqslant}

\renewcommand{\phi}{\varphi}

\newcommand{\ab}{\operatorname{ab}}
\newcommand{\y}{\mathbf{y}}
\newcommand{\xra}{\xrightarrow}
\newcommand{\tra}{\operatorname{tra}}
\newcommand{\id}{\operatorname{id}}

\newcommand{\RV}{\operatorname{RV}}

\newcommand{\N}{\mathcal{N}}
\newcommand{\F}{\mathcal{F}}

\newcommand{\A}{\mathcal{A}}

\newcommand{\Q}{\mathcal{Q}}

\renewcommand{\k}{\operatorname{\mathbf{k}}}

\newcommand{\gen}[1]{\langle #1 \rangle}

\newcommand{\CC}{\mathbb{C}}
\newcommand{\FF}{\mathbb{F}}

\begin{document}

\title{Weight conjectures for fusion systems on an extraspecial group}

\author{Radha Kessar}
\address{Department of Mathematics, University of Manchester, Manchester M13 9PL,
  United Kingdom}
\email{radha.kessar@manchester.ac.uk}

\author{Markus Linckelmann}
\address{Department of Mathematics, City St Georges, University of London EC1V 0HB, 
United Kingdom}
\email{markus.linckelmann.1@city.ac.uk}

\author{Justin Lynd}
\address{Department of Mathematics, University of Louisiana at 
Lafayette, Lafayette, LA 70504}
\email{lynd@louisiana.edu}

\author{Jason Semeraro}
\address{Department of Mathematical Sciences, Loughborough University, Loughborough LE11 3TU,
  United Kingdom}
\email{j.p.semeraro@lboro.ac.uk}

\begin{abstract} 
In a previous paper, we stated and motivated counting conjectures for fusion
systems that are purely local analogues of several local-to-global conjectures
in the modular representation theory of finite groups.  Here we verify some of
these conjectures for fusion systems on an extraspecial group of order $p^3$,
which contain among them the Ruiz-Viruel exotic fusion systems at the prime
$7$. As a byproduct we verify Robinson's ordinary weight conjecture for principal $p$-blocks of almost simple groups $G$ realizing such (nonconstrained) fusion systems.
\end{abstract}



\maketitle

\section{Introduction}\label{intro} 

Let $k$ be an algebraically closed field of characteristic $p > 0$, let $G$ be a finite group, and let $B$ be a block of $kG$.
One way to encode the local structure of $B$ is via a pair $(\F,\alpha)$, where $\F = \F_S(B)$ is the fusion system of the block on the defect group $S$ of $B$ and $\alpha = (\alpha_{Q})_{Q \in \F^c}$ is a certain compatible family of $2$-cohomology classes of automorphism groups of centric subgroups of the fusion system. 
In this context, an Alperin weight is a projective simple module for the twisted group algebra $k_{\alpha}\Out_\F(Q)$,
and so the number $\w(B) = \w(\F,\alpha)$ of $B$-weights can be counted from the information in the pair.

If $\F$ is a saturated fusion system on a finite $p$-group, there is nothing preventing one from fixing a compatible family $\alpha$ and imagining $(\F,\alpha)$ arose from a block in the above fashion, even if this is not actually the case. 
In the paper \cite{KLLS2019}, we took this point of view.
There we defined two numbers $\k(\F,\alpha)$ and $\m(\F,\alpha)$ that are both stand-ins for the number of ``ordinary irreducible characters'' of $(\F,\alpha)$. If $(\F,\alpha)$ comes from a block $B$, then $\m(\F,\alpha)$ is exactly the local side of the equality appearing in a certain ``summed-up'' version of Robinson's ordinary weight conjecture (SOWC), which is supposed to count the number of ordinary irreducible characters in $B$ (see the paragraph preceding \cite[Conjecture 2.3]{KLLS2019}).
On the other hand, $\k(\F,\alpha)$ is the local side of an equality that is an immediate consequence of Alperin's weight conjecture (AWC) and which is supposed to count the same number. 
It was shown in \cite{KLLS2019} that $\k(\F,\alpha) = \m(\F,\alpha)$ for all pairs $(\F,\alpha)$, possibly with $\F$ an exotic fusion system, conditional on the validity of AWC for all blocks of finite group algebras.
Indeed, our proof of this is a fusion theoretic version of Robinson's theorem that AWC holds for all blocks if and only if SOWC does. 

With Alperin's and Robinson's weight conjectures as a bridge from local to global, one can formulate purely local analogues of several of the local-to-global counting conjectures in block theory, and
there seems to be no reason why these should not hold also for exotic pairs. 
It was shown in \cite{LyndSemeraro2023} that the exotic Benson-Solomon fusion systems at the prime $2$ have only the trivial compatible family, and some of the conjectures were verified for these systems in \cite{Semeraro2021}. 
The recent paper \cite{KSST} verifies six of the conjectures for fusion systems on a Sylow $p$-subgroup of $G_2(p)$, a class which contains 27 exotic systems. 

In this paper
we show that a saturated fusion system $\F$ on an extraspecial $p$-group $S$ of order $p^3$ and exponent $p$ supports just the trivial compatible family (Theorem~\ref{l:kp-extraspecial}) and that many of the purely local counting conjectures hold for $\F$ (for the list, see Conjecture~\ref{sixconjectures}).
Along with the other fusion systems on extraspecial groups, like those coming from sporadic groups, we give counts of the above representation theoretic quantities (as well as others that take into account defects of characters) for the exotic Ruiz-Viruel fusion systems \cite{RuizViruel2004}.
For example, we show that $\m(RV_1,0) = 41$ and $\m(RV_2,0) = 33$, i.e., that a $7$-block with a simple Ruiz-Viruel exotic fusion system at the prime $7$ would have $41$ or $33$ ordinary irreducible characters, respectively, were such a block to exist (it doesn't \cite{KessarStancu2008}).  Furthemore,  in Proposition  \ref{prop:verfiyowc} we combine our results  with those  of \cite{NarasakiUno2009} to observe that  OWC holds  for the principal $p$-blocks of all almost simple groups $G$ which realise a nonconstrained fusion system $\F$ on an extraspecial group of order $p^3$ and exponent $p$. In \cite{eaton2004equivalence}, Eaton shows that OWC is equivalent to the
Dade projective conjecture in the sense that a minimal counterexample to one is a minimal counterexample to the other. This latter conjecture has previously been verified for principal $p$-blocks of some of the groups  considered  above (see, e.g., \cite{an1998alperin, an2003alperin}, \cite{Narasaki2007}), and in fact for all cases at the prime $3$ (see \cite{usami2001principal}).

\textbf{Acknowledgements:} We thank Robin Allen for help  with the proof of Lemma \ref{l:chardeg}  and  Shigeo Koshitani for  bringing the  paper of Narasaki and Uno to our attention. M. Linckelmann  gratefully acknowledges funding from the
UK Research Council EPSRC for the project  EP/X035328/1.  J. Semeraro gratefully acknowledges funding from the
UK Research Council EPSRC for the project EP/W028794/1.
J. Lynd thanks City St. George's, University of London, and Peter Symonds at the University of Manchester for support and hospitality during research visits where some of the work on this paper took place. 

\section{Some weight conjectures for fusion systems}

We recall in this section just enough detail to state seven conjectures that we verify for nonconstrained fusion systems on extraspecial $p$-groups, and refer to \cite{KLLS2019} and \cite[Sections~8.13-8.15]{LinckelmannBT2} for more details and additional motivation. 
For notation involving fusion systems we follow \cite{AschbacherKessarOliver2011}.
As before, we fix an algebraically closed field $k$ of positive characteristic $p$. 
Let $\F$ be a saturated fusion system on a finite $p$-group $S$ and $\F^c$ the set of $\F$-centric subgroups of $S$. 
Assume given a compatible family $\alpha = (\alpha_Q)_{Q \in \F^c}$ of $2$-cohomology classes $\alpha_Q \in H^2(\Out_\F(Q),k^\times)$ with coefficients in the multiplicative group of $k$, see \cite[Definition~4.1]{KLLS2019} or \cite[Theorem~8.14.5]{LinckelmannBT2}. 
If $\F$ is the fusion system of a block, there is a canonical compatible family of \emph{K\"ulshammer--Puig} classes of this form. 

Denote by $k_{\alpha_Q}\Out_\F(Q)$ the twisted group algebra, and define
\[
\w(\F,\alpha) = \sum_{Q \in \F^c/\F} z(k_{\alpha_Q}\Out_\F(Q))
\]
where $z(-)$ denotes the number of isomorphism classes of 
projective simple modules for an algebra and the sum runs over a set of representatives for the $\F$-conjugacy classes of centric subgroups $Q$. 
Thus, $\w(\F,\alpha)$ is the number of Alperin $B$-weights if $(\F,\alpha)$ comes from a block $B$ \cite[Theorem~8.14.4]{LinckelmannBT2}, so $B$ satisfies AWC just when $\ell(B) = \w(\F,\alpha)$. 

Set
\[
\k(\F,\alpha) = \sum_{x \in S/\F} \w(C_\F(x),\alpha(x)),
\]
where $C_\F(x)$ is the centralizer subsystem, $\alpha(x)$ is the restriction of $\alpha$ along the inclusion functor $C_\F(x)^c \to \F^c$, and the sum runs over fully centralized representatives for $\F$-conjugacy classes of elements.
If $(\F,\alpha)$ comes from a block $B$, then using a lemma of Brauer about decomposing conjugacy classes by $p$-sections, the block satisfies AWC just when the number $\mathbf{k}(B)$ of ordinary irreducible characters in $B$ is $\k(\F,\alpha)$.

Finally, we define quantities motivated by Robinson's Ordinary Weight Conjecture (OWC) \cite{Robinson1996}.
For a nonnegative integer $d$ and a subgroup $Q$ of $S$, write $\Irr_d(Q)$ for the ordinary characters $\mu$ of defect $d$, that is, such that the $p$-part of $|Q|/\mu(1)$ is $p^d$.
Let $\N_Q$ be the set of normal chains $\sigma =(1 = X_0 < X_1 < \cdots < X_n)$ in $\Out_\F(Q)$, those such that $X_i$ is normal in $X_n$ for each $i$, and which begin with the trivial subgroup.
The length of $\sigma$ as displayed is $n$. 
The sets $\N_\Q$ and $\Irr_K^d(Q)$ are $\Out_\F(Q)$-invariant, and we let $I(\sigma)$ and $I(\sigma, \mu)$ be the stabilizers in $\Out_\F(Q)$ of $\sigma$ and of the pair $(\sigma,\mu)$, respectively.

For a pair $(\F,\alpha)$ and $d \geq 0$, set 
\[
w_{Q}(\F,\alpha,d) = \sum_{\sigma \in \N_Q/\Out_\F(Q)} (-1)^{|\sigma|} \sum_{\mu \in \Irr_K^d(Q)/I(\sigma)} z(k_\alpha I(\sigma,\mu)),
\]
and 
\[
\m(\F,\alpha,d) = \sum_{Q \in \F^c/\F} w_Q(\F,\alpha,d). 
\]
If $(\F,\alpha)$ arises from a block $B$, then OWC is the statement that the number $\mathbf{k}_d(B)$ of ordinary characters of defect $d$ is equal to $\m(\F,\alpha,d)$. 

We can now state the seven conjectures from \cite{KLLS2019} that we will verify
for nonconstrained fusion systems on an extraspecial group. 

\begin{Conj}\label{sixconjectures}
Let $\F$ be a saturated fusion system on finite $p$-group $S$ of order $p^e$, and let $\alpha$
be a compatible family. Then
\begin{enumerate}[label=\textup{(\arabic*)}, ref=\arabic*]
\item \label{c:k=m} $\k(\F,\alpha) = \m(\F,\alpha)$.
\item \label{conj:k(b)} $\k(\F,\alpha) \leq |S|$.
\item \label{c: malle-robinson} $\w(\F,\alpha) \leq p^s$, where $s$ is the sectional rank of $S$.
\item \label{c:defect} For each positive integer $d$, we have $\m(\F,\alpha,d) \ge 0$.
\item \label{c:heightzero} If $S$ is nonabelian, then $\m(\F, \alpha, d') \ne 0$ for some $d \ne e$. 
\item \label{c:eatonmoreto} If $S$ is nonabelian and $r > 0$ is the smallest positive integer such
that $S$ has a character of degree $p^r$, then $r$ is the smallest positive
integer such that $\m(\F, \alpha, d-r) \ne 0$.  
\item \label{c:malle-navarro} 
\begin{enumerate}[label=\textup{(\alph*)}, ref=\alph*]
\item $\k(\F, \alpha)/\m(\F, \alpha, e)$ is at most the number of conjugacy 
classes of $[S,S]$.
\item $\k(\F, \alpha)/\w(\F, \alpha)$ is at most the number of 
conjugacy classes of $S$.
\end{enumerate}
\end{enumerate}
\end{Conj}

If $(\F,\alpha)$ comes from a block $B$ then \eqref{c:k=m}, respectively \eqref{c:defect}, holds if AWC, respectively OWC, holds for $B$ as well as for all of its Brauer correspondents. Moreover each of the other statements 
in Conjecture \ref{sixconjectures} hold if and only if a known conjecture holds for the block $B$.
Those conjectures are, respectively, \eqref{conj:k(b)} the $\mathbf{k}(B)$-conjecture, \eqref{c: malle-robinson} the Malle--Robinson conjecture \cite{MalleRobinson2017}, \eqref{c:heightzero} Brauer's height zero conjecture, \eqref{c:eatonmoreto} the Eaton--Moreto conjecture \cite{EatonMoreto2014}, and \eqref{c:malle-navarro} the Malle--Navarro conjecture \cite{MalleNavarro2006}. 

We conclude this section by showing that \eqref{conj:k(b)} would follow as a consequence of \eqref{c:malle-navarro}(a).

\begin{Prop}\label{p:kgv}
Let $\F$ be a saturated fusion system on a finite $p$-group $S$ of order $p^e$ and let $\alpha$ be a compatible family. Then $$\m(\F,\alpha,e) \le |S:S'|.$$ 
In particular Conjecture \ref{sixconjectures}\eqref{c:malle-navarro}(a) implies Conjecture \ref{sixconjectures}\eqref{conj:k(b)}.
\end{Prop}

\begin{proof}
Since $|\mathcal{N}_S|=1$, we have $$\m(\F,\alpha,e)=\sum_{\mu \in \Irr^e(S)/\Out_\F(S)} z(k_{\alpha_S}C_{\Out_\F(S)}(\mu)). $$
Let $\F_0 =  N_{\F}(S)$  and let   $E$  be a   complement  to  $\Inn (S)$  in $\Aut_{\F}(S)$.  Then  $S \rtimes E $ is a model for  $\F_0$. Moreover, there exists  a central  extension, say $G$, of $  S \rtimes E $  by a $p'$-group and a  block $B$ of $kG$ with fusion system  $\F_0$  and  K\"ulshammer-Puig cocycle at $S$  equal to $\alpha_S$ (see Prop. IV.5.35 of \cite{AschbacherKessarOliver2011} and its proof). By \cite[Theorem 2.5]{Robinsonweight2004}, the ordinary weight conjecture holds for $B$, that is the number of maximal defect characters in $B$ is
$$ \sum_{\mu \in \Irr^e(S)/\Out_{\F_0}(S)} z(k_{\alpha_S}C_{\Out_{\F_0}(S)}(\mu))$$ which is  $\m(\F,\alpha,e)$ since $\Out_{\F_0}(S)=\Out_\F(S)$. Thus by \cite[Theorem 1]{kulshammer1987remark} the inequality $\m(\F,\alpha,e) \le |S:S'|$ is a consequence of Brauer's $\k(B)$-conjecture for $p$-solvable groups, or equivalently the $k(GV)$-conjecture (see \cite{nagao1962conjecture}). Since this latter statement is known by \cite{gluck2004solution}, the result holds. Finally, if Conjecture \ref{sixconjectures}\eqref{c:malle-navarro}(a) holds then  $\k(\F,\alpha) \le |S'^{cl}| \m(\F,\alpha,e) \le |S'|\m(\F,\alpha,e) \le |S|$ and so Conjecture \ref{sixconjectures}\eqref{conj:k(b)} also holds. 
\end{proof}

\section{Fusion systems on an extraspecial group of order $p^3$ }
\label{exampleSection} 

In this section we recall the Ruiz-Viruel classification \cite{RuizViruel2004}
of the saturated fusion systems over an extraspecial group $S$ of order $p^3$
and exponent $p$ while setting up notation.  There are three exotic fusion
systems over $S$ when $p = 7$.  

For the remainder of this section, $S \cong p_+^{1+2}$ denotes an 
extraspecial group of order $p^3$ and exponent $p$. Fix the presentation
\begin{eqnarray}
\label{e:p3pres}
\langle \a,\b,\c \mid \a^p=\b^p=\c^p=[\a,\c]=[\b,\c]=[\a,\b]\c^{-1}=1\rangle
\end{eqnarray}
for $S$. Each element of $S$ can then be written in the form $\a^r\b^s\c^t$,
for unique nonnegative integers $0 \leq r,s,t \leq p-1$. Let $\epsilon$ be a
fixed primitive $p$'th root of $1$ in $\mathbb{C}$. 
We first list some basic information about $S$, along with general information
about the associated saturated fusion systems on $S$. 

\begin{Lem}\label{l:exsp-setup}
Let $S$ be an extraspecial $p$-group of order $p^3$ and of exponent $p$, as
given above.

\begin{enumerate}
\item The map which sends $X:=\left(\begin{smallmatrix} x & y \\ z & w
\end{smallmatrix} \right)$ to the class $[\varphi] \in \Out(S)$ of the
automorphism $\varphi$ given by $$\a \longmapsto \a^x\b^z, \hspace{4mm} \b
\longmapsto \a^y\b^w, \hspace{4mm}  \c  \longmapsto \c^{\det(X)},$$ is an
isomorphism from $\GL_2(p)$ to $\Out(S)$.
\item A complete set of $S$-conjugacy class representives of elements 
of $S$ is given by 
\[
\{\a^i\b^j \mid 0 \le i,j \le p-1 \} \cup \{\c^k \mid 1 \le k \le p-1\}.
\] 
\item $\Irr(S)$ consists of $p^2$ linear characters $\chi_{u,v}$, $0
\le u,v \le p-1$, and $p-1$ faithful characters $\phi_u$,
$1 \le u \le p-1$, of degree $p$. The characters are given
explicitly by 
\[
\chi_{u,v}(\a^r\b^s\c^t):= \epsilon^{ru+sv} \quad\text{ and }\quad
\phi_u(\a^r\b^s\c^t):=
\begin{cases} 
p\epsilon^{ut} & \mbox{ if $r=s=0$} \\ 0 & \mbox{ otherwise.} 
\end{cases}.
\]
\end{enumerate}
\end{Lem}

\begin{proof}
See \cite{RuizViruel2004}, for example, for (1) and (2). Part (3) is contained
in \cite[Theorem~5.5.4]{Gorenstein1980}.   
\end{proof}

The elementary abelian subgroups of $S$ of order $p^2$ are in one-to-one
correspondence with the points of the projective line over 
$\mathbb{F}_p$; set
\[
Q_i:=\langle \c,\a\b^i \rangle  
\hspace{4mm} (0 \le i \le p-1) \hspace{4mm} \mbox{ and } \hspace{4mm} 
Q_p:=\langle \c,\b \rangle.
\]

From \cite[Section III.6.2]{AschbacherKessarOliver2011}, we have the 
following description of the fusion systems on $S$.

\begin{Thm}\label{t:classp3}
Let $\F$ be a saturated fusion system on $S$. Then 
\[
\F^{cr} \subseteq \{Q_i \mid 0 \le i \le p\} \cup \{S\}
\]
and the following hold:
\begin{enumerate}
\item $\Out_\F(S)$ is a $p'$-group.
\item For each $i$ and $\alpha \in \Aut_\F(S)$ such that $\alpha(Q_i)=Q_i,
\alpha|_{Q_i} \in \Aut_\F(Q_i)$.
\item If $Q_i \in \F^r$, then $\SL_2(p) \le \Aut_\F(Q_i) \le \GL_2(p)$. If $Q_i
\notin \F^r$ then $\Aut_\F(Q_i)=\{\varphi|_{Q_i} \mid \varphi \in
N_{\Aut_\F(S)}(Q_i)\}$.
\item If $Q_i \in \F^r$ and $\beta \in N_{\Aut_\F(Q_i)}(Z(S))$, then $\beta$
extends to an element of $\Aut_\F(S)$.
\end{enumerate}
Conversely, any fusion system $\F$ over $S$ for which conditions (1)-(4) hold
and for which each morphism is a composition of restrictions of
$\F$-automorphism of $S$ and the $Q_i$ is saturated. Moreover, in this case,
$\F$ is constrained if and only if at most $1$ of the $Q_i$ is centric and
radical.
\end{Thm}

When $\F$ is nonconstrained, the possibilities for $\Out_\F(S)$ and
$\Out_\F(Q_i)$ (and hence $\F$) are given in 
\cite[Tables 1.1 and 1.2]{RuizViruel2004}. 
Apart from the $p$-fusion systems of $\PSL_3(p)$, $p \ge 3$
and their almost simple extensions, there are thirteen exceptional fusion
systems for $3 \le p \le 13$, three of which are exotic. 

\section{Compatible families}
The goal of this section is to show that there are no nonzero compatible
families for the nonconstrained fusion systems over $S$.  
Recall that the Schur multiplier $M(G)$ of a finite group $G$ is a finite
abelian group, which may be defined as the second cohomology group
$H^2(G,\mathbb{C}^\times)$.  In computing the Schur multiplier of various
groups, we make use of its connection with stem extensions.  A \textit{stem
extension} of a group $G$ is a central extension $1 \to Z \to \hat{G} \xra{\pi}
G \to 1$ such that $\ker(\pi) \leq Z(\hat{G}) \cap [\hat{G},\hat{G}]$. We will
often identify $\ker(\pi)$ with $Z$.  If $Z \cong M(G)$ in this situation, then
the extension (or $\pi$, or $\hat{G})$ is said to be a \textit{Schur covering}
of $G$.  
Given a central extension as above, there is an associated
inflation-restriction exact sequence
\begin{eqnarray}
\label{e:LHS}
1 \to H^{1}(G,\mathbb{C}^\times) \xra{\inf} H^1(\hat{G},\mathbb{C}^\times)
\xra{\res} \Hom(Z,\mathbb{C}^\times) \xra{\tra} H^2(G,\mathbb{C}^\times)
\xra{\inf} H^2(\hat{G},\mathbb{C}^\times) 
\end{eqnarray}
in which three of the maps are given by inflation or restriction, and the
fourth is the transgression map. This is defined by first choosing a cocycle
$\alpha$ representing the class $[\alpha] \in H^2(G,Z)$ of the extension.  For
any homomorphism $\phi \in \Hom(Z,\mathbb{C}^\times)$, post-composition with
$\phi$ yields a 2-cocycle with values in $\mathbb{C}^\times$, and then
$\tra(\phi)$ is defined as the class $[\phi \circ \alpha] \in H^2(G,
\mathbb{C}^\times)$. 

The next lemma collects a number of general results regarding the Schur
multiplier which will be used later in special cases. The first four 
parts are due to Schur. In parts (4) and (5) results of Schur and of 
Blackburn \cite{Blackburn1972} are quoted, and require the following 
additional notation. Denote the abelianization of a group $G$ by 
$G^{\ab}$, and write $G^{\ab} \wedge G^{\ab}$ for the quotient of 
$G^{\ab} \otimes_{\mathbb{Z}} G^{\ab}$ by the subgroup generated by 
$a \otimes b + b \otimes a$ as $a$ and $b$ range over $G^{\ab}$. 

\begin{Lem}\label{l:coverings}
The following hold for a finite group $G$.
\begin{enumerate}
\item If $1 \to Z \to \hat{G} \to G \to 1$ is a stem extension, then 
the transgression map $\tra$ in $\eqref{e:LHS}$ is injective, and 
hence $Z \cong \Hom(Z,\mathbb{C}^\times)$ is isomorphic to a subgroup 
of $M(G)$. 
\item There exists a Schur covering $1 \to Z \to \hat{G} \xra{\pi} G \to 1$ 
and, for any such covering, the associated transgression map is an 
isomorphism.
\item If there exists a Schur covering as in (2) such that $\pi\colon
\pi^{-1}(H) \to H$ is a stem extension of some subgroup $H$ of $G$, then the
restriction map $M(G) \to M(H)$ is injective.  
\item If $G = G_1 \times G_2$ is a direct product, then one has 
\[
M(G) \cong M(G_1) \times M(G_2) \times (G_1^{\ab} 
\otimes_{\mathbb{Z}} G_2^{\ab}). 
\]
\item (Blackburn) Assume $G = K \wr H$ is a wreath product, and write $m$ for
the number of involutions in $H$. Then $M(G)$ is isomorphic to the direct
product of $M(H)$, $M(K)$, $\frac{1}{2}(|H|-m-1)$ copies of $K^{\ab} \otimes
K^{\ab}$, and $m$ copies of $K^{\ab} \wedge K^{\ab}$.
\end{enumerate}
\end{Lem}

\begin{proof}   We refer to \cite[Lemma~11.42]{CurtisReiner1990} for (1) and to
\cite[Theorem~11.43]{CurtisReiner1990} for (2).  Assume the hypotheses of 
(3) and set $\widehat{H} = \pi^{-1}(H)$. Choose a $2$-cocycle $\alpha$ 
representing the class of the central extension 
$1 \to Z \to \hat{G} \to G \to 1$. Then $\alpha|_{H\times H}$ represents 
the class of the central extension $1 \to Z \to \hat{H} \to H \to 1$, 
and the square 
\[
\xymatrix{
\Hom(Z,\mathbb{C}^\times) \ar[r]^{\tra_G} \ar[d]_{\id} & 
H^2(G,\mathbb{C}^\times) \ar[d]^{\res}\\
\Hom(Z,\mathbb{C}^\times) \ar[r]_{\tra_H} & H^2(H,\mathbb{C}^{\times})
}
\]
commutes. As $\tra_H$ is injective by (1) and $\tra_G$ an isomorphism 
by (2), part (3) follows. Finally, a  proof  of  (4)  may be found in  \cite[Corollary~3]{Wiegold71}   and  (5) is \cite[Theorem~1]{Blackburn1972}  but see also \cite[Theorem~2.2.10,  Theorem~6.3.3]{Karpilovsky1987}.
\end{proof}

Whenever $p$ is a prime, we write $M(G)_p$ for the $p$-primary part of $M(G)$,
and $M(G)_{p'}$ for the $p'$-primary part of $M(G)$.  The following lemma
collects some basic information about the various primary parts of the Schur
multiplier. 

\begin{Lem}\label{l:cohom}
Let $G$ be a finite group, let $p$ be a prime, and let $k$ be an algebraically
closed field of characteristic $p$.
\begin{enumerate}
\item $H^2(G,k^\times)$ is isomorphic to the $p'$-part $M(G)_{p'}$ of the Schur
multiplier. 
\item If $H \leq G$ contains a Sylow $p$-subgroup of $G$, then the restriction
map $M(G)_p \to M(H)_p$ is injective. 
\item If $G$ has cyclic Sylow $p$-subgroups, then $M(G)_p = 1$.
\end{enumerate}
\end{Lem}

\begin{proof}
See \cite[Proposition~2.1.14]{Karpilovsky1987} for part (1). Part (2) is
derived from the fact that restriction to $H$ followed by transfer to $G$ is
multiplication by the index of $H$ in $G$, which by assumption is prime to $p$.
Then part (3) follows from (2) and the fact that   the second cohomology group of a finite cyclic group   with coefficients in a divisible group with trivial action is trivial    but see also\cite[Proposition~11.46]{CurtisReiner1990}.
\end{proof}

We now specialize to the following computations of Schur multipliers of
specific finite groups that appear as subgroups of automorphism groups of
centric radicals in certain nonconstrained saturated systems over $S$. 

\begin{Lem}\label{l:mg}
The following hold. 
\begin{enumerate}
\item If $p$ is any prime and $G$ is a subgroup of $\GL_2(p)$ containing
$\SL_2(p)$, then $M(G) = 1$. 
\item Let $G$ be a $2$-group of maximal class. Then $M(G) \cong C_{2}$ if $G$
is dihedral, while $M(G) = 1$ if $G$ is semidihedral or quaternion.  If $G$ is
a dihedral $2$-group and $V$ is any four subgroup of $G$, then the restriction
$M(G) \to M(V)$ is injective. 
\item Let $G = C_n \wr C_2$ with $n \geq 2$. Then $M(G) \cong C_2$ if $n$ is
even, and $M(G) = 1$ if $n$ is odd.  If $J \leq G$ is the homocyclic subgroup
of rank $2$ and exponent $n$, then the restriction $M(G) \to M(J)$ is
injective. 
\item If $G = \Out_{\F}(S)$ for some fusion system $\F$ over $S$ appearing in
Table~1.2 of \cite{RuizViruel2004}, then $M(G)_{2'} = 1$. Moreover, either
$M(G)_2 = 1$, or $M(G)_2 \cong C_2$ and $G$ is $D_8$, $S_3 \times C_6$, $C_6
\wr C_2$, $D_8 \times C_3$, or $D_{16} \times C_3$.
\end{enumerate}
\end{Lem}

\begin{proof}
The fact that $\SL_2(p)$ has trivial multiplier for all primes $p$ is 
standard: note that all Sylow $r$-subgroups of $\SL_2(p)$ are cyclic, 
except when $p \geq 3$ and $r = 2$, in which case a Sylow $r$-subgroup 
is generalized quaternion. It follows that $M(G)_r = 1$ for all primes 
$r$ by (2) below and Lemma~\ref{l:cohom}(2). Also $\GL_2(3)$ also has 
trivial multiplier by (2) and Lemma~\ref{l:cohom}(2) since a Sylow 
$2$-subgroup of $\GL_2(3)$ is semidihedral. So to finish the proof of 
(1), we may assume that $p \geq 5$. Let $G$ be a subgroup of $\GL_2(p)$ 
containing $N = \SL_2(p)$. Since $p \geq 5$, $N$ is perfect. Fix a 
group $\widehat{G}$ having a central subgroup $Z$ such that 
$\widehat{G}/Z \cong G$.  Identify $\widehat{G}/Z$ with $G$ and let
$\pi \colon \widehat{G} \to G$ be the canonical projection.  We shall 
show that there is a complement to $Z$ in $\widehat{G}$. Let 
$\widehat{N}$ be the preimage of $N$ under $\pi$.  Then $\widehat{N}$ 
contains $Z$ and splits over it, as $M(N) = 1$. Fix any complement 
$N_0$ of $Z$ in $\widehat{N}$, so that $\widehat{N} = N_0 \times Z$.  
Then $N_0 \cong \SL_2(p)$. We claim that $N_0$ is normal in 
$\widehat{G}$; it is clear that $N_0$ is normal in $\widehat{N}$. In 
general, conjugation by $g \in \widehat{G}$ sends an element $h \in N_0$ 
to $h'\zeta_g(h)$, where $h' \in N_0$ and $\zeta_g(h) \in Z$ are
uniquely determined. Also since $Z$ is central in $\widehat{G}$, the 
assignment $h \mapsto \zeta_g(h)$ is a group homomorphism from $N_0$ to 
$Z$.  Since $N_0$ is perfect, it follows that $\zeta_g = 1$ for each 
$g \in \widehat{G}$, i.e. $N_0$ is normal in $\widehat{G}$. Write 
quotients by $N_0$ with pluses. Now $\widehat{G}^+$ is a central 
extension of $Z$ by $\widehat{G}/\widehat{N} \cong G/N$, which is a 
cyclic $p'$-group. On the other hand, we have $M(G/N) = 1$ by
Lemma~\ref{l:cohom}(3), applied with each prime divisor $r$ of $|G/N|$ 
in the role of ``$p$'' there.  Thus, we may fix a complement $K^+$ of 
$Z^+$ in $\widehat{G}^+$, and let $K$ be the preimage of $K^+$ in 
$\widehat{G}$.  Then $K$ is a complement to $Z$ in $\widehat{G}$. 

Part (2) is implied by Lemma~\ref{l:coverings} as follows. Let $G$ be 
a $2$-group of maximal class, so that $G$ is dihedral, semidihedral or
quaternion. Then $M(G)$ is a $2$-group.  Let $\pi \colon \hat{G} \to G$ 
be any Schur covering of $G$ with $Z = \ker(\pi)$. Then since 
$Z \leq [\hat{G}, \hat{G}]$, it follows that 
$\hat{G}^{\ab} \cong G^{\ab}$ is of order $4$. By
\cite[Theorem~5.4.5]{Gorenstein1980}, $\hat{G}$ is of maximal class, so
$Z(\hat{G})$ is of order $2$. Then either $M(G) = 1$, or 
$Z(\hat{G}) = Z$ and $M(G) = C_2$. In the latter case, since $\hat{G}$ 
is of maximal class, we have $\hat{G}/Z \cong G$ is dihedral.  
Conversely, the dihedral group $\hat{G} = D_{2^{k+1}}$ provides a Schur 
covering $\pi\colon \hat{G} \to G$ of $G = D_{2^{k}}$.  Fix a four 
subgroup $V$ of $G$.  Then as $\pi^{-1}(V)$ is dihedral of order $8$, 
we have that the restriction map $M(G) \to M(V)$ is injective by
Lemma~\ref{l:coverings}(3). This completes the proof of (2).

To prove (3), apply Lemma~\ref{l:coverings}(5) with $K = C_n$ and 
$H = C_2$. Hence, $m = 1$ there. By Lemma~\ref{l:cohom}(3) and that 
result, $M(G) = K^{\ab} \wedge K^{\ab}$. The multiplication map 
$C_n \otimes C_n \to C_n$ is an isomorphism, where $C_n$ is viewed as 
an additive group, and under that map $a \otimes b + b \otimes a$ is 
sent to $2ab$.  Hence, $M(G) \cong C_n/2C_n \cong C_2$ if $n$ is even, 
and $1$ if $n$ is odd.  In order to prove the claim about the 
restriction $J$, we may by Lemma~\ref{l:cohom}(2) assume that $n = 2^l$ 
for some $l$, and then it suffices by Lemma~\ref{l:coverings}(3) to 
produce a double covering of $G$ which restricts to a stem extension of 
$J$.  To this end, the group $G = C_{2^l} \wr C_2$ has a presentation 
with generators $x$, $y$, and $t$, and defining relations 
$x^{2^l} = y^{2^l} = t^2 = xyx^{-1}y^{-1} = 1$ and $txt^{-1} = y$.  
Consider the group $\hat{G}$ with generators $\x, \y, \t$ and defining 
relations $\x^{2^l} = \y^{2^l} = \t^2 = 1$, $\t\x\t^{-1} = \y$, and 
$\z = [\x,\y]$ is of order $2$ and central.  Thus 
$\hat{G}/\gen{\x^{2^{l-1}},\y^{2^{l-1}}} \cong D_{16}$, and the obvious 
map $\pi\colon \hat{G} \to G$ is the pullback of the Schur covering 
$D_{16} \to D_{8}$. Let $\mathbf{J}$ be the preimage of $J$ in 
$\hat{G}$, and set $Z = \ker(\pi) = \gen{\z}$. Then by construction 
$Z \leq [\mathbf{J},\mathbf{J}] \cap Z(\mathbf{J})$, so 
$\pi\colon \mathbf{J} \to J$ is a stem extension of $J$. As noted 
above, this completes the proof of (3).

We now prove (4).  When $G$ is $D_8$, $SD_{16}$, $C_6 \wr C_2$, 
$S_3 \times C_3$, $S_3 \times C_6$, $D_8 \times C_3$, 
$D_{16} \times C_3$, or $SD_{32} \times C_3$, the claim follows from 
(1), (2), (3), and Lemma~\ref{l:coverings}(4).  It remains to consider 
the groups ``$4S_4$'' and ``$C_3 \times 4S_4$'' in 
\cite[Table~1.2]{RuizViruel2004}.  Note that $4S_4$ as appears in 
\cite{RuizViruel2004} is the normalizer $G$ in $\GL_2(5)$ of a Sylow
$2$-subgroup $Q \cong Q_8$ of $\SL_2(5)$.  The normalizer $N$ in 
$\SL_2(5)$ of $Q$ is the commutator subgroup of $G$, isomorphic to 
$\SL_2(3)$, and the quotient $G/N$ is cyclic of order $4$. Thus, 
$G^{\ab}$ is cyclic of order $4$. Therefore, it suffices to show that 
$M(G) = 1$, for then by Lemma~\ref{l:coverings}(4), we'll have 
$M(C_3 \times G) = 1$.  Since $G$ has Sylow $3$-subgroups of order $3$, 
it follows from Lemma~\ref{l:cohom}(3) that $M(G) = M_2(G)$.  Since 
$\SL_2(3) \cong Q \rtimes C_3$ is $2$-perfect and $G/N$ is cyclic, the 
exact same argument as given in (1) applies with $Z$ a $2$-group to 
show that $M(G)_2 = 1$.  This completes the proof of (4) and the lemma.
\end{proof}

\begin{Thm}\label{l:kp-extraspecial}
Let $p$ be an odd prime, and let $\F$ be a nonconstrained saturated fusion
system on an extraspecial $p$-group $S$ of order $p^3$ and of exponent $p$.
Then $\lim \A^2_\F = 0$. 
\end{Thm}
\begin{proof}
Consider the cochain complex $(C^*(\A_\F^2), \delta)$ computing the limits of
$\A_\F^2$ as in \cite{Linckelmann2009}. By Lemma \ref{l:mg}(1),
$M(\Out_\F(Q))_{p'} = 1$ for all elementary abelian subgroups $Q \in \F^{cr}$.
Thus, the $0$-th cochain group is $C^0(\A_\F^2) = M(\Out_\F(S))_{p'} =
M(\Out_\F(S))$, and the coboundary map 
\[
\delta^0 \colon M(\Out_\F(S)) \longrightarrow 
\bigoplus_{Q \in \F^{cr}, |Q| = p^2} M(\Out_\F([Q < S]))
\]
is the sum of the restriction maps $M(\Out_\F(S)) \to M(\Out_\F([Q<S]))$.
Thus, to complete the proof it suffices to show that at least one of these
restriction maps is injective.

We regard $\Out_\F(S) \leq \GL_2(p)$ as acting on $\{Q_i \mid  0 \leq i \leq
p\}$ as it does on the projective line, and then $\Out_{\F}([Q < S])$ is the
stabilizer of the point $Q$. We go through the possibilities for $\F$ appearing
in Tables~1.1 and 1.2 of \cite{RuizViruel2004}. Consider first a fusion system
$\F$ over $S$ occuring in Table~1.2. Then $C^0(\A_\F^2) = 1$ unless $G :=
\Out_\F(S)$ appears in Lemma \ref{l:mg}(4), and in those cases there exists
some $Q \in \F^{cr}$ with $|Q|=p^2$ such that $M(\Out_\F([Q < S]))$ is a four
subgroup.  Hence, $\ker(\delta^0) = 1$ by Lemma \ref{l:mg}(2), and so $\lim
\A_\F^2 = 0$ in this case. Now consider a nonconstrained fusion system
appearing in Table~1.1. Then $\F^{cr} = \{Q_0,Q_p\}$, and $G := \Out_\F(S)$ may
be taken in the normalizer in $GL_2(p)$ of the subgroup $T$ of diagonal
matrices, which stabilizes $Q_0$ and $Q_p$. Write $G_0 = \Out_\F([Q_0<S])$, for
short. Then $G_0 = G \cap T$.  Assume first that $r \neq p$ is an odd prime. If
$r$ divides $p+1$, then a Sylow $r$-subgroup of $GL_2(p)$ is cyclic, and hence
$M(G)_r = 1$ by Lemma~\ref{l:cohom}(3). Suppose $r$ divides $p-1$.  Then a
Sylow $r$-subgroup of $G$ is contained in $T$, and so $G_0$ contains a Sylow
$r$-subgroup of $G$.  Thus the restriction $M(G)_r \to M(G_0)_r$ is injective
by Lemma \ref{l:cohom}(2). It remains to consider $r = 2$.  Inspection of
Table~1.1 of \cite{RuizViruel2004} shows that either a Sylow $2$-subgroup of
$G$ stabilizes $Q_0$, in which case $M(G)_2 \to M(G_0)_2$ is injective by
Lemma~\ref{l:cohom}(2), or a Sylow $2$-subgroup $R$ of $G$ is isomorphic to
$C_{2^l} \wr C_2$ for some $l \geq 1$, and $R \cap G_0$ is the homocyclic
subgroup of $R$ of index $2$.  The restriction map $M(G)_2 \to M(G_0)_2$ is
injective in this latter case by Lemma \ref{l:mg}(3).  
\end{proof}

\section{Verification of the conjectures}
We now verify a number of the conjectures in Section~\ref{intro} for a
nonconstrained saturated fusion system $\F$ over $S$.  In the thirteen 
exceptional cases complete proofs by hand could be written down, but 
since we have no reasonably general argument for the specific numerical 
computations, the conjectures are ultimately verified using computer
calculations in Magma \cite{Magma}. 

First we need the following well-known
lemma.

\begin{Lem}\label{l:sl2p}
Let $G$ be a finite group with normal subgroup $N$. Suppose that $V$ is an
inertial projective simple $kN$-module and that $G/N$ is a cyclic $p'$-group.
Then $G$ has exactly $|G:N|$ projective simple modules lying over $V$. In
particular, if $q$ is a power of $p$ and $N = \SL_n(q) \le G \le \GL_n(q)$,
then $z(kG)=|G:\SL_n(q)|$.
\end{Lem}
\begin{proof}
Since $G/N$ is a cyclic group, we have that $H^2(G/N,k^\times)$ is 
trivial. Thus, the hypotheses of Corollary 5.3.13 of \cite{LinckelmannBT1} 
hold. By that result and its proof, we may fix a simple $kG$-module $U$ 
with $\Res_N^G(U) \cong V$, and then the collection of isomorphism 
classes of one dimensional $kG/N$-modules is in one-to-one correspondence 
with the collection of simple $kG$-modules restricting to $V$ via the 
map $W \mapsto U \otimes_k W$, where we
regard $W$ as a module for $G$ by inflation. Also, the $U \otimes_k W$ are
precisely the summands of the induced module $\Ind_N^G V$, and hence are all
projective. This completes the proof of the first statement. 

In the special case of the last statement, by results of  Steinberg  
(see  \cite[Theorem~3.7, Theorem~8.3]{Humphreys2006})  $N = \SL_n(p)$ 
has exactly one projective simple module, the Steinberg module, which is 
therefore inertial. Since $G/N$ is a cyclic $p'$-group in this case, and 
since any projective $kG$-module is projective also as a $kN$-module, the 
last statement follows.
\end{proof}

In a saturated fusion system $\F$ on $S$, Lemma~\ref{l:exsp-setup}(a) and
Lemma~\ref{t:classp3}(1) allow one to identify $\Out_\F(S)$ with a subgroup of
$\GL_2(p)$ of order prime to $p$.  Relative to this setup, we will identify
elements of $\Out_\F(S)$ with matrices with respect to the basis $\{\a,\b\}$
(or rather, the basis $\{\a Z(S), \b Z(S)\}$ of $S^{\ab}$).  Write
$\Out^*_\F(S)$ for those elements of $\Out_\F(S)$ which have determinant $1$.
The next lemma gives a general calculation of the quantity $\m(\F,0,d)$. 

\begin{Lem}\label{l:det=1}
Let $\F$ be a saturated fusion system on $S$.  Then
\begin{enumerate}
\item $\m(\F,0,0) = \m(\F,0,1) = 0$,
\item $\m(\F,0,2) = \frac{p-1}{l} \cdot |\Out_\F^*(S)^{\cl}|$, where $l$
denotes the index of $\Out_\F^*(S)$ in $\Out_\F(S)$, and 
\item $\m(\F,0,3) = \sum_{\mu} z(kC_{\Out_\F(S)}(\mu))$, where $\mu$ runs over
a set of representatives for the $\Out_\F(S)$-orbits of linear characters of
$S$, and where $C_{\Out_\F(S)}(\mu)$ denotes the stabilizer of $\mu$ in
$\Out_\F(S)$. 
\end{enumerate}
Thus, 
\[
\m(\F,0) = \frac{p-1}{l} \cdot |\Out^*_\F(S)^{\cl}| + 
\sum_{\mu \in \Irr^3(S)/\Out_{\F}(S)} z(kC_{\Out_\F(S)}(\mu)).
\]
\end{Lem}

\begin{proof}
Each character of $Q_i \cong C_p \times C_p$ is linear, so has defect $2$,
while each character of $S$ has defect $2$ or $3$ by
Lemma~\ref{l:exsp-setup}(3).  This shows (1).

The quantity $\m(\F,0,2)$ is a sum over $\F$-conjugacy classes of centric
radicals $Q$ of the quantities $\w_Q(\F,0,2)$.  Fix first a centric radical $Q$
of order $p^2$. We claim $\w_{Q}(\F,0,2) = 0$.  By Theorem~\ref{t:classp3},
$\Out_\F(Q)$ is a subgroup of $\GL_2(p)$ containing $\SL_2(p)$ with index $a$,
say. First consider the trivial chain $\sigma = (\bar{Q}) \in \N_Q$. Then
$\Out_\F(Q)$ stabilizes $\sigma$. There are two orbits on characters, one nontrivial and one
trivial. Given a nontrivial character $\mu$, the stabilizer of $\mu$ in
$\Out_\F(Q)$ is a Frobenius group with normal subgroup of order $p$, and so has
$0$ projective simple modules by \cite[Lemma~4.11]{KLLS2019}.  Also, the trivial
character has stabilizer $\Out_\F(Q)$, which by Lemma~\ref{l:sl2p} has $a$
projective simple modules. Thus, the contribution to $\w_Q(\F,0,2)$ from the
trivial chain is $a$. Next let $\sigma = (\bar{Q} < \bar{S}) \in \N_Q$ be the
nontrivial chain. The stabilizer $I(\sigma)$ in $\Out_\F(Q)$ of $\sigma$ has a
normal subgroup $\bar{S}$ and the quotient by $\bar{S}$ is abelian order
$(p-1)a$.  Then $I(\sigma)$ has one orbit of size $1$ containing the trivial
character, and one of size $p-1$ consisting of those characters with kernel
$Z(S)$. The respective stabilizers have $\bar{S}$ as a normal subgroup, so the
contribution from these orbits is $0$, by \cite[Lemma~4.11]{KLLS2019}.  The group
$\bar{S}$ acts regularly on the remaining points of $\mathbb{P}^1(Q)$, and a
subgroup in $I(\sigma) \cap SL(Q)$ of order $p-1$ acts regularly on the
characters with kernel a given point. Thus, there is one further orbit with
stabilizer of order $a$, and this contributes $-a$ to $\w_{Q}(\F,0,2)$. Hence,
$\w_Q(\F,0,2) = a-a = 0$, as claimed.

Lastly, consider $\w_{S}(\F,0,2)$. Here, there is only the trivial chain with
stabilizer $\Out_\F(S)$. Note that $\Irr^2(S)$ consists of the $p-1$ faithful
characters denoted $\phi_u$ in Lemma~\ref{l:exsp-setup}. Let $l$ be the index
of $\Out^*_\F(S)$ in $\Out_\F(S)$. The action of $\Out_\F(S)$ on $\Irr^2(S)$ is
the same as the action on $Z(S)^\#$ (nonidentity elements), and there
$\Out_\F(S)$ acts semiregularly via the determinant map. So there are $(p-1)/l$
orbits, each of size $l$, and a representative for each orbit has stabilizer
$\Out^*_\F(S)$.  As $\Out^*_{\F}(S)$ is a $p'$-group, we have
$z(k\Out^*_{\F}(S))$ is the number of simple $k\Out^*_\F(S)$-modules, which is
the number of conjugacy classes.  This completes the proof of (2).

In (3), since $\Irr^3(Q_i)$ is empty, there is only one pair $(Q,\sigma)$ for
which the innermost sum in $\m(\F,0,3)$ is nonzero, namely the pair $(S, 1)$.
We thus have $I(\sigma) = \Out_\F(S)$ for this pair. Thus, $$\m(\F,0,3) =
\sum_{\mu \in \Irr^3(S)/\Out_\F(S)} z(kC_{\Out_\F(S)}(\mu)),$$ which completes
the proof of (3). The last statement then follows from (1)-(3), since clearly
$\m(\F,0,d) = 0$ for $d > 3$.
\end{proof}

We now calculate $\k(\F,0)$.

\begin{Lem}\label{l:k}
Let $\F$ be a saturated fusion system on $S$.  Then 
$$\k(\F,0)=\w(\F,0)+\frac{p-1}{l} \cdot |\Out^*_\F(S)^{\cl}|
+\sum_{([x] \in S^{\cl}\backslash \{1,\c\})/\F} z(kC_{\Out_\F(S)}([x])).$$
\end{Lem}

\begin{proof}
By definition, we have
$$\k(\F,0)= \w(\F,0)+\w(C_\F(\c),0)+
\sum_{[x] \in (S^{\cl}\backslash \{1,\c\})/\F} \w(C_\F(x),0),$$ 
where in the latter sum $[x]$ runs over $S$-classes for which 
$\langle x \rangle$ is fully $\F$-centralized. Let $1 \neq x \in S$ be 
such that $x$ is not $\F$-conjugate to $\c$ and $\langle x \rangle$ is 
fully $\F$-centralised. If $Q \le C_S(x)$ is $C_\F(x)$-centric radical, 
then $|Q|=p^2$ and $x \in C_S(Q)=Q$. Therefore 
$Q = \langle x, \c \rangle$ and  $\Out_{C_\F(x)}(Q)$ does not contain a 
copy of $\SL_2(p)$, a contradiction. We conclude that 
$C_S(x) \unlhd C_\F(x)$, and 
$$\w(C_\F(x),0)=z(k\Out_{C_\F(x)}(C_S(x))=
z(kC_{\Out_\F(C_S(x))}(x)),$$ 
by \cite[Lemma~3.3]{KLLS2019}. We claim that 
$C_{\Out_\F(C_S(x))}([x]) \cong C_{\Out_\F(S)}([x])$. Since $\F$ is saturated, 
the restriction map 
$$\res: C_{\Aut_\F(S)}(x) \rightarrow 
C_{\Aut_\F(C_S(x))}(x)$$ 
is surjective with kernel containing $C_{\Inn(S)}(x)$ as a Sylow 
$p$-subgroup. If $A$ is a $p'$-subgroup of $\ker(\res)$ then
$A$ commutes with both $C_{\Inn(S)}(x) \cong \langle x \rangle$ and $C_S(C_{\Inn(S)}(x))$. Hence Thompson's $A \times B$ Lemma \cite[Theorem~5.3.4]{Gorenstein1980} implies that $A=1$, and 
$C_{\Out_\F(S)}([x]) \cong C_{\Aut_\F(S)}(x)/C_{\Inn(S)}(x) 
\cong C_{\Out_\F(C_S(x))}(x)$ as claimed.

Clearly $S \unlhd C_\F(\c)$ and so 
$$\w(C_\F(\c),0)=z(k\Out_{C_\F(\c)}(S))=
z(kC_{\Out_\F(S)}([\c]))=\frac{p-1}{l} \cdot |\Out^*_\F(S)^{\cl}|,$$ 
where $l$ is as given in Lemma \ref{l:det=1}. This completes the proof.
\end{proof}

\begin{Rem}
From the classification in \cite{RuizViruel2004} we see that for any 
nonconstrained fusion system $\F$ over $S$, one has $l = p-1$ in 
Lemmas~\ref{l:det=1} and \ref{l:k} . Indeed, if $\F$ is nonconstrained, then $Z(S) =
ZJ(S)$ is not weakly $\F$-closed, and then a result of Glauberman then suggests
that $\Out_{\F}(Z(S)) \cong C_{p-1}$ for any nonconstrained fusion system over
$S$; c.f.  \cite[Theorem~14.14]{Glauberman1971}.
\end{Rem}

\begin{Lem}\label{l:someconjclasses}
Let $\omega$ be a generator of the multiplicative group
$\mathbb{F}_p^\times$.
\begin{enumerate}
\item The wreath product $ \langle
\left(\begin{smallmatrix} \omega & 0 \\ 0 & 1 
\end{smallmatrix} \right), 
\left(\begin{smallmatrix} 1 & 0 \\ 0 & \omega \end{smallmatrix} \right), 
\left(\begin{smallmatrix} 0 & 1 \\ 1 & 0 \end{smallmatrix} \right)\rangle \cong C_{p-1} \wr C_2$ 
has $(p-1)(p+2)/2$ conjugacy classes.
\item $ \langle \left(\begin{smallmatrix} \omega & 0 \\ 0 & \omega^{-1} 
\end{smallmatrix} \right), 
\left(\begin{smallmatrix} 0 & -1 \\ 1 & 0 \end{smallmatrix} \right) \rangle \cong C_{p-1}.C_2$ 
has $(p+5)/2$ conjugacy classes.
\item If $3 \mid p-1$ then 
$\langle \left(\begin{smallmatrix} \omega^3 & 0 \\ 0 & 1 
\end{smallmatrix} \right), 
\left(\begin{smallmatrix} \omega & 0 \\ 0 & \omega 
\end{smallmatrix} \right), 
\left(\begin{smallmatrix} 0 & 1 \\ 1 & 0 \end{smallmatrix} \right) \rangle$ 
has $(p-1)(p+8)/6$ conjugacy classes.
\item If $3 \mid p-1$ then 
$\langle \left(\begin{smallmatrix} \omega^3 & 0 \\ 0 & \omega^{-3} 
\end{smallmatrix} \right), 
\left(\begin{smallmatrix} 0 & -1 \\ 1 & 0 \end{smallmatrix} \right) \rangle$ 
has $(p+17)/2$ conjugacy classes.
\end{enumerate}
\end{Lem}

\begin{proof}
If $B$ denotes the base of the wreath product $G:=C_n \wr C_2$, then $Z(G)$ is
a cyclic subgroup of order $n$ in $B$ so $B \backslash Z(G)$ is the union of
$(n^2-n)/2$ classes. There are $n$ classes of elements in $G$ outside
$B$ which yields $n(n+3)/2$ classes altogether and (1) holds. Next, if $G$
denotes the group in (2) and $H$ is the cyclic subgroup of order $p-1$, we see
that, apart from $Z(G)$ (of order $2$), there are $(p-3)/2$ classes of elements
in $H$. There are $2$ classes of elements outside of $H$, which yields
$(p+5)/2$ classes altogether. A similar argument proves (4). Finally, we prove
(3). Let $G$ denote the group in question, and set  
$B:=\langle \left(\begin{smallmatrix} \omega^3 & 0 \\ 0 & 1 
\end{smallmatrix} \right), 
\left(\begin{smallmatrix} \omega & 0 \\ 0 & \omega 
\end{smallmatrix} \right) \rangle$. 
We see that $Z(G)$ has order $p-1$ and so $B \backslash Z(G)$ is the union of
$(p-1)(p-4)/6$ classes. There are $p-1$ classes of elements in $G$ outside $B$
and this yields $(p-1)(p+8)/6$ classes altogether, as needed. 
\end{proof}

We now describe the computations of the quantities $\m(\F,0)$, $\w(\F,0)$, 
and $\k(\F,0)$ that were carried out in Magma \cite{Magma} and listed in
Tables~\ref{t:rv2} and \ref{t:rv3}.  The list of nonconstrained saturated
fusion systems on $S$ is given in Table~\ref{t:rv2}, based on the list in
\cite[Tables~1.1, 1.2]{RuizViruel2004}. Generators for $\Out_{\F}(S)$ are
listed in the third column of Table~\ref{t:rv2} for the convenience of 
the reader: in each case there is exactly one $\Out(S)$-conjugacy class 
of subgroups isomorphic with $\Out_\F(S)$, and a representative is 
chosen to contain as many diagonal matrices as possible. Then 
$\Out_\F(S)$-orbit representatives and stabilizers for the actions on 
linear characters of $S$ and $S$-conjugacy classes were computed and 
listed in columns four through seven, using the notation of
Lemma~\ref{l:exsp-setup}. In each case, $\m(\F,0,3)$ is computed using 
Lemma~\ref{l:det=1}(3), by summing up the number of projective simple
modules of the stabilizers 
listed in the fifth column of Table \ref{t:rv2}.  Then $\m(\F,0,3)$
is listed in Table~\ref{t:rv3}.  The quantity $\m(\F,0,2)$ is computed
using Lemma~\ref{l:det=1}(2) and Lemma~\ref{l:someconjclasses}, which computes
the number of conjugacy classes of the various $\Out_\F^*(S)$ listed in
Table~\ref{t:rv3}. This completes the description of the computation of
$\m(\F,0)$ as the sum of $\m(\F,0,2)$ and $\m(\F,0,3)$.  Then $\w(\F,0)$ is
calculated using the list of outer automorphism groups of centric radicals in
\cite[Tables 1.1 and 1.2]{RuizViruel2004}. For example, we have denoted 
the three exotic fusion systems at the prime $7$ by
$\RV_{1}$, $\RV_2$, and $\RV_2:2$ in the tables, where $\Out_{\RV_1}(S) \cong
C_6 \wr C_2$, $\Out_{\RV_2}(S) \cong D_{16} \times C_3$, and $\Out_{\RV_2:2}(S)
\cong SD_{32} \times C_3$. Note that $\RV_{2}:2$ contains $\RV_2$ as a normal
subsystem of index $2$. These systems have the following invariants.
\begin{itemize}
\item $\m(\RV_1,0) = 41$ and $\w(\RV_1,0) = 35$, 
\item $\m(\RV_2,0) = 33$ and $\w(\RV_2,0) = 25$, and
\item $\m(\RV_2:2,0) = 42$ and $\w(\RV_2:2,0) = 35$. 
\end{itemize} 

Finally, $\k(\F,0)$ is calculated using Lemma \ref{l:k} by adding 
$\w(\F,0)$ and $\frac{p-1}{l}|\Out_\F^*(S)^{\cl}|$ to the sum of the number of 
projective simple modules of the stabilizers listed in the seventh 
column of Table \ref{t:rv2}.

\begin{Prop}\label{p:main}
Let $\F$ be a nonconstrained saturated fusion system on $S$. 
Then Conjectures~\ref{sixconjectures}\eqref{c:k=m}-\eqref{c:malle-navarro} all hold for $\F$.
\end{Prop}
\begin{proof}
This can be easily verified at this point using the tables.
\end{proof}

Our final result, Proposition \ref{prop:verfiyowc} is a verification of OWC for principal $p$-blocks of all almost simple groups which realise $\F$. We require the following result concerning character degrees of extensions of $\PSL_3(p)$. 

\begin{Lem}\label{l:chardeg}
For $p \ge 3$ the character degrees of $\PSL_3(p)$, $\PGL_3(p)$, $\PSL_3(p):2$ and $\PGL_3(p):2$ are listed in Table \ref{t:degrees}.
\end{Lem}

\begin{proof}
The character degrees for $\PSL_3(p)$ are obtained using \cite[Table 2]{simpson1973character} and those for $\PGL_3(p)$ are given in \cite[Table VIII]{steinberg1951representations}. This accounts for the degrees in columns $2,4,6$ of Table \ref{t:degrees}.    The entries   in  columns $3$ and $7$  are respectively computed from those  of columns $2$ and $6$ in the following way.  The group in columns $2$ and $6$ is $\PGL_3(p) $ and the  group in columns $3$ and $7$ is   $\Aut(\PSL_3(p))= \PGL_3(p) \langle  \tau \rangle   $ with  $\tau $  acting as   the transpose inverse automorphism.  Since  every $\GL_3(p)$-conjugacy class  is   invariant under taking transpose,   an irreducible character  of     $\PGL_3(p) $   is   $\tau$-stable if and only if    it is real-valued.  Thus, by Clifford theory, if  an irreducible character $\chi$  of $\PGL_3(p)  $  is real then it is covered by two characters of $\Aut(\PSL_3(p))$    of the same degree as  $\chi $ and if $\chi$ is not real it is covered by a  unique character of $\Aut(\PSL_3(p))$  of degree twice that of $\chi$.

It thus suffices to determine the real characters of $\PGL_3(p)$. By \cite[Section 5.2]{fulton2013representation}, the characters of $\GL_2(p)$ comprise: \begin{itemize}
\item $p-1$ linear characters $\psi_\alpha$, for $\alpha \in \widehat{\FF_p^\times} = \Hom(\FF_p^\times, \CC^\times)$;
\item $p-1$ characters $\st_\alpha$ of degree $p$, for $\alpha \in \widehat{\FF_p^\times}$;
\item ${p-1 \choose 2}$ characters $\chi_{\alpha,\beta}$ of degree $p+1$ for distinct $\alpha, \beta \in \widehat{\FF_p^\times}$;
\item ${p \choose 2}$ characters $\chi_\varphi$ of degree $p-1$ for $\varphi \in \widehat{\FF_{p^2}^\times}$ with $\varphi \neq \varphi^p$.
\end{itemize}

For $\chi \in \Irr(\GL_2(p))$ and $\psi \in \widehat{\FF_p^\times}$ let $\chi \boxtimes \psi$ denote the character of $\GL_3(p)$ obtained via parabolic induction. Let $\varphi^1$, $\varphi^{p(p+1)}$ and $\varphi^{p^3}$ denote the unipotent characters of $\GL_3(p)$ and set $\varphi^{(-)}_\alpha:=\varphi^{(-)} \otimes \alpha$ for $\alpha \in \widehat{\FF_p^\times}$. Further let $\iota \in \widehat{\FF_p^\times}$ be of order $2$. In terms of these descriptions, column $4$ of Table \ref{t:realclasses} lists  $2p+6$ distinct real characters of $\GL_3(p)$ and so these must account for all real characters by  \cite{gill2011real}. Those which descend to $\PGL_3(p)$ are described in column 6. This explains the degrees listed in columns $3$ and $7$ of Table \ref{t:degrees}.  

The entries in column $5$ can be determined  in the following fashion. Let $N=\PSL_3(p)$,  $G$  be as in column $5$ and  $G_1= \PGL_3(p)$.
Let $\chi \in \Irr(N)$.   By Clifford theory we have the following: 
\begin{itemize}
\item If $\chi$  is $G_1$-stable  and covered by a   real character of $G_1$, then  $\chi  $ is $G$-stable (and hence extends to two characters of $G$).
\item If $\chi$ is  $G_1$-stable  but not covered by any real  character   of $G_1$, then $\chi$ is not $G$-stable. 
\item   If $\chi$ is  not $G_1$-stable but  is covered  by a  real character of $G_1$, then  one character in the  $G_1$-orbit of $\chi $ is $G$-stable and the other two are permuted  by $G$.
\item  If  $\chi $ is  not $G_1$-stable  and is not covered by a real character of $G_1$, then   no character in the  $G_1 $-orbit of   $\chi $ is  $G$-stable.
\end{itemize}
     Which of the four cases applies can be determined from the information  in   \cite{simpson1973character} and \cite{steinberg1951representations} and from  Table 1.  \end{proof}
 
\begin{table}
\tiny{

\renewcommand{\arraystretch}{1.6}
\centering
\caption{Real characters of $\PGL_3(p)$}
\label{t:realclasses}
\begin{tabular}{|c|c|c|c|c|c|c|}
\hline
$\chi \in \GL_3(p)$ & \mbox{notes} & $\chi(1)$  & \mbox{real?} & $\#$  $\mbox{real}$ & $\mbox{descend to } \PGL_3(p)?$ & $\#$ \mbox{descend} \\
\hline
$\varphi_\alpha^1$ & $\alpha \in \widehat{\FF_p^\times}$ & $1$  & $\alpha =1,\iota$ & $2$ & $\alpha=1$ & $1$ \\ \hline $\varphi_\alpha^{p(p+1)}$ & $\alpha \in \widehat{\FF_p^\times}$ & $p(p+1)$  &  $\alpha=1,\iota$ & $2$ & $\alpha=1$ & $1$ \\ \hline $\varphi_\alpha^{p^3}$ & $\alpha \in \widehat{\FF_p^\times}$ & $p^3$  & $\alpha=1,\iota$ & $2$ & $\alpha=1$ & $1$ \\ \hline $\psi_\alpha \boxtimes \beta$ & $\alpha,\beta \in  \widehat{\FF_p^\times}, \alpha \neq \beta$ & $p^2+p+1$ &  $\alpha,\beta = 1, \iota$ & $2$ & $\alpha=\iota, \beta =1$ & $1$ \\ \hline
$\st_\alpha \boxtimes \beta$ & $\alpha,\beta \in  \widehat{\FF_p^\times}, \alpha \neq \beta$ & $p(p^2+p+1)$  & $\alpha,\beta = 1, \iota$ & $2$ & $\alpha=\iota, \beta =1$ & $1$ \\ \hline $\chi_{\alpha,\beta} \boxtimes \gamma$ & $\alpha,\beta,\gamma \in  \widehat{\FF_p^\times},\mbox{ distinct}$ & $(p+1)(p^2+p+1)$  & $\alpha =\overline{\beta}, \gamma = 1, \iota$ & $p-3$ & $\gamma=1$ & $\frac{p-3}{2}$ \\ \hline $\chi_\varphi \boxtimes \alpha$ & $\varphi^p \neq \varphi \in  \widehat{\FF_{p^2}^\times}, \alpha \in \widehat{\FF_p^\times}$ & $(p-1)(p^2+p+1)$  & $\varphi|_{\FF_p^\times}=1,\alpha=1,\iota$ & $p-1$ & $\alpha=1$ & $\frac{p-1}{2}$ \\ \hline
\end{tabular}
}
\end{table}

\begin{Prop}\label{prop:verfiyowc}
Let $G$ be a finite almost simple group with extraspecial Sylow $p$-subgroups of order $p^3$ and nonconstrained fusion system at the odd prime $p$.  The principal $p$-block of $G$ satisfies Robinson's Ordinary Weight Conjecture.
\end{Prop}

\begin{proof}
Let $N$ be a finite simple group and $N \leq G \leq \Aut(N)$.
Fix a Sylow $p$-subgroup $S \cong p_+^{1+2}$ of $G$ with $p$ odd, and set $\F = \F_S(G)$. 
If $|G/N|$ is divisible by $p$, then there can be at most one radical elementary abelian subgroup of order $p^2$, and so $\F$ is constrained. 
We may thus assume that $S \leq N$.  If  $ p=3 $, then the  result follows by comparing Table \ref{t:rv3} with  the Table  after Proposition 63 of \cite{NarasakiUno2009}. 
Now suppose  $p\geq 5 $.  If $N$ is sporadic, then  by  \cite[Remark~1.4]{RuizViruel2004},  $G$ is one of $\operatorname{Th}$, $\operatorname{Ru}$ $(p=5)$, $\operatorname{He}$,  $\operatorname{He}.2$,  $\operatorname{Fi}'_{24}$,
$\operatorname{Fi}_{24}$, $\operatorname{O}'\operatorname{N}$, $\operatorname{O}'\operatorname{N}.2$ $(p=7)$, $\operatorname{M}$ $(p=13)$. Character degrees  for these groups  and their partition into blocks are recovered from GAP using the \texttt{PrimeBlocks} command and  the result  follows by  comparison with Table \ref{t:rv3}.  So we may assume that $N$  is not sporadic.  By  \cite[Theorem 31]{NarasakiUno2009},    $N$ is one of  $\PSL_3(p)$ or  $\PSU_3(p)$. The normalizer of a Sylow $p$-subgroup in $\PSU_3(p)$ is strongly $p$-embedded by \cite[Theorem~7.6.2(a)]{GLS3}, and so $S$ is normal in $\F_S(N)$ in this case.  Thus, $N$  is   $\PSL_3(p)$  (and  this has a  nonconstrained  $p$-fusion system).  By  Lemma \ref{l:chardeg} the character degrees of $G$ are listed in Table \ref{t:degrees}, and those for its principal $p$-block are obtained by excluding rows $6$ and $7$ from this list.  Again, the result follows by comparing with Table \ref{t:rv3}.
\end{proof}

\begin{landscape}
\small
\begin{table}
\renewcommand{\arraystretch}{1.5}
\centering
\caption{$\Out_\F(S)$-orbits of $\Irr^3(S)$, $\F$-classes of $S^{\cl}$ 
and  their $\Out_\F(S)$-stabilisers }
\label{t:rv2}
\begin{tabular}{|c|c|c|c|c|c|c|}

\hline
$p$ & $\F$ & $\Out_\F(S)$ & $\Irr^3(S)/\Out_\F(S)$ & stabilisers 
& $(S^{\cl} \backslash \{1,\c\})/\F$ & stabilisers \\ 
\hline

$3 \nmid (p-1)$ & $\PSL_3(p)$ & $\langle \left(\begin{smallmatrix} \omega 
& 0 \\ 
0 & 1 \end{smallmatrix} \right), \left(\begin{smallmatrix} 1 & 0 \\ 
0 & \omega \end{smallmatrix} \right) \rangle$ 
& $\chi_{0,0},\chi_{0,1},\chi_{1,0},\chi_{1,1}$ 
& $C_{p-1}^2,C_{p-1},C_{p-1},1$ &$[\a\b]$&$1$\\ 
\hline

$3 \nmid (p-1)$ & $\PSL_3(p):2$ 
& $\langle \left(\begin{smallmatrix} \omega & 0 \\ 
0 & 1 \end{smallmatrix} \right), 
\left(\begin{smallmatrix} 1 & 0 \\ 0 & \omega \end{smallmatrix} \right), 
\left(\begin{smallmatrix} 0 & 1 \\ 1 & 0 \end{smallmatrix} \right) \rangle$ 
& $\chi_{0,0},\chi_{0,1},\chi_{1,1}$ & $C_{p-1} \wr C_2 ,C_{p-1},C_2$ 
&$[\a\b]$&$C_2$\\ 
\hline

$3 \mid (p-1)$ & $\PSL_3(p)$ 
& $\langle \left(\begin{smallmatrix} \omega^3 & 0 \\ 
0 & 1 \end{smallmatrix} \right), \left(\begin{smallmatrix} \omega & 0 \\ 
0 & \omega \end{smallmatrix} \right) \rangle$  
& $\makecell{\chi_{0,0},\chi_{0,1},\chi_{1,0},\\
\chi_{1,1},\chi_{\omega,1},  \chi_{1,\omega}}$ 
& $\makecell{C_{p-1} \times C_{(p-1)/3}, C_{(p-1)/3}, \\
C_{(p-1)/3},1,1,1}$  &$[\a\b]$, $[\a\b^\omega]$, $[\a^\omega \b]$ 
&$1$, $1$, $1$\\ 
\hline

$3 \mid (p-1)$ & $\PSL_3(p):2$ 
& $\langle \left(\begin{smallmatrix} \omega^3 & 0 \\ 
0 & 1 \end{smallmatrix} \right), 
\left(\begin{smallmatrix} \omega & 0 \\ 0 & \omega \end{smallmatrix} \right), 
\left(\begin{smallmatrix} 0 & 1 \\ 1 & 0 \end{smallmatrix} \right) \rangle$  
& $\makecell{\chi_{0,0},\chi_{0,1},\\ \chi_{1,1},\chi_{\omega,1}}$ 
& $\makecell{(C_{p-1} \times C_{(p-1)/3}):C_2, \\ C_{(p-1)/3},C_2,1}$  
&$[\a\b]$, $[\a\b^\omega ]$ & $C_2$, $1$ \\ 
\hline

$3 \mid (p-1)$ & $\PSL_3(p):3$ 
& $\langle \left(\begin{smallmatrix} \omega & 0 \\ 
0 & 1 \end{smallmatrix} \right), 
\left(\begin{smallmatrix} 1 & 0 \\ 
0 & \omega \end{smallmatrix} \right) \rangle$ 
& $\chi_{0,0},\chi_{0,1},\chi_{1,0},\chi_{1,1}$ 
& $C_{p-1}^2,C_{p-1},C_{p-1},1$ & $[\a\b]$ & $1$ \\ 
\hline

$3 \mid (p-1)$ & $\PSL_3(p):S_3$ 
& $\langle \left(\begin{smallmatrix} \omega & 0 \\ 
0 & 1 \end{smallmatrix} \right), 
\left(\begin{smallmatrix} 1 & 0 \\ 0 & \omega \end{smallmatrix} \right), 
\left(\begin{smallmatrix} 0 & 1 \\ 1 & 0 \end{smallmatrix} \right) \rangle$ 
& $\chi_{0,0},\chi_{0,1},\chi_{1,1}$ & $C_{p-1} \wr C_2 ,C_{p-1},C_2$ 
& $[\a\b]$ & $C_2$ \\ 
\hline

$3$ & $^2F_4(2)'$ 
& $\langle \left(\begin{smallmatrix} -1 & 0 \\ 
0 & 1 \end{smallmatrix} \right), 
\left(\begin{smallmatrix} 1 & 0 \\ 0 & -1 \end{smallmatrix} \right), 
\left(\begin{smallmatrix} 0 & 1 \\ 1 & 0 \end{smallmatrix} \right) \rangle$ 
& $\chi_{0,0},\chi_{0,1},\chi_{1,1}$ &$D_8$, $C_2$, $C_2$ & $-$ & $-$ \\ 
\hline

$3$ & $\J_4$ 
& $\langle \left(\begin{smallmatrix} 2 & 0 \\ 
0 & 1 \end{smallmatrix} \right), 
\left(\begin{smallmatrix} 1 & 0 \\ 0 & 2 \end{smallmatrix} \right), 
\left(\begin{smallmatrix} 1 & 2 \\ 2 & 2 \end{smallmatrix} \right) \rangle$ 
& $\chi_{0,0},\chi_{0,1}$ &$SD_{16}$, $C_2$  &$-$ & $-$ \\ 
\hline

$5$ & $\Th$ & $\langle \left(\begin{smallmatrix} 2 & 0 \\ 
0 & 1 \end{smallmatrix} \right), 
\left(\begin{smallmatrix} 1 & 0 \\ 0 & 2 \end{smallmatrix} \right), 
\left(\begin{smallmatrix} 3 & 3 \\ -1 & 1 \end{smallmatrix} \right) \rangle$ 
&$\chi_{0,0},\chi_{0,1}$ &$4.S_4,C_4$ &$-$ & $-$ \\ 
\hline

$7$ & $\He$ & $\langle \left(\begin{smallmatrix} 2 & 0 \\ 
0 & 1 \end{smallmatrix} \right), 
\left(\begin{smallmatrix} 1 & 0 \\ 0 & 2 \end{smallmatrix} \right), 
\left(\begin{smallmatrix} 0 & -1 \\ -1 & 0 \end{smallmatrix} \right) \rangle$ 
& $\makecell{\chi_{0,0},\chi_{0,1},\chi_{1,0},\\
\chi_{1,1},\chi_{3,1},  \chi_{1,3}}$ 
&$\makecell{S_3 \times C_3, C_3 \\ 
C_3,1,C_2,C_2}$ &$[\a]$, $[\b]$, $[\a\b^3]$, $[\a^3 \b]$ 
& $C_3$, $C_3$, $C_2$, $C_2$ \\ 
\hline

$7$ & $\He:2$ & $\langle \left(\begin{smallmatrix} 2 & 0 \\ 
0 & 1 \end{smallmatrix} \right), 
\left(\begin{smallmatrix} 3 & 0 \\ 0 & 3 \end{smallmatrix} \right), 
\left(\begin{smallmatrix} 0 & 1 \\ 1 & 0 \end{smallmatrix} \right) \rangle$ 
&  $\chi_{0,0}, \chi_{0,1}, \chi_{1,1},\chi_{3,1}$ 
&$S_3 \times C_6, C_3, C_2, C_2 $ & $[\a]$, $[\a\b^3]$ & $C_3$, $C_2$ \\ 
\hline

$7$ & $\Fi'_{24}$ & $\langle \left(\begin{smallmatrix} 2 & 0 \\ 
0 & 1 \end{smallmatrix} \right), 
\left(\begin{smallmatrix} 3 & 0 \\ 0 & 3 \end{smallmatrix} \right), 
\left(\begin{smallmatrix} 0 & -1 \\ -1 & 0 \end{smallmatrix} \right) \rangle$ 
& $\chi_{0,0},\chi_{0,1},\chi_{1,1},  \chi_{3,1}$ 
&$S_3 \times C_6$, $C_3$, $C_2$, $C_2$ &$[\b]$ & $C_3$ \\ 
\hline

$7$ & $\Fi_{24}$ & $\langle \left(\begin{smallmatrix} 3 & 0 \\ 
0 & 1 \end{smallmatrix} \right), 
\left(\begin{smallmatrix} 1 & 0 \\ 0 & 3 \end{smallmatrix} \right), 
\left(\begin{smallmatrix} 0 & 1 \\ 1 & 0 \end{smallmatrix} \right) \rangle$    
& $\chi_{0,0},\chi_{0,1},\chi_{1,1}$  
& $C_6 \wr C_2$, $C_6$, $C_2$  &$[\b]$ & $C_6$\\ 
\hline

$7$ & $\RV_1$ & $\langle \left(\begin{smallmatrix} 3 & 0 \\ 
0 & 1 \end{smallmatrix} \right), 
\left(\begin{smallmatrix} 1 & 0 \\ 0 & 3 \end{smallmatrix} \right), 
\left(\begin{smallmatrix} 0 & 1 \\ 1 & 0 \end{smallmatrix} \right) \rangle$    
& $\chi_{0,0},\chi_{0,1},\chi_{1,1}$ 
& $C_6 \wr C_2$, $C_6$, $C_2$  & $-$ & $-$ \\ 
\hline

$7$ & $\ON$ & $\langle \left(\begin{smallmatrix} 3 & 0 \\ 
0 & 3 \end{smallmatrix} \right), 
\left(\begin{smallmatrix} 1 & 0 \\ 0 & -1 \end{smallmatrix} \right), 
\left(\begin{smallmatrix} 0 & 2 \\ 3 & 0 \end{smallmatrix} \right) \rangle$ 
& $\chi_{0,0},\chi_{0,1},\chi_{1,1},  \chi_{1,3}$ 
& $D_8 \times C_3, C_2,1,C_2$ & $[\a\b]$ & $1$ \\ 
\hline

$7$ & $\ON:2$ & $\langle \left(\begin{smallmatrix} 3 & 0 \\ 
0 & 3 \end{smallmatrix} \right), 
\left(\begin{smallmatrix} 1 & 0 \\ 0 & -1 \end{smallmatrix} \right), 
\left(\begin{smallmatrix} 2 & 4 \\ -1 & 2 \end{smallmatrix} \right) \rangle$
& $\chi_{0,0},\chi_{0,1},\chi_{1,1}$ 
& $D_{16} \times C_3$, $C_2$, $C_2$ & $[\a\b]$ & $C_2$ \\ 
\hline

$7$ & $\RV_2$ & $\langle \left(\begin{smallmatrix} 3 & 0 \\ 
0 & 3 \end{smallmatrix} \right), 
\left(\begin{smallmatrix} 1 & 0 \\ 0 & -1 \end{smallmatrix} \right), 
\left(\begin{smallmatrix} 2 & 4 \\ -1 & 2 \end{smallmatrix} \right) \rangle$ 
& $\chi_{0,0},\chi_{0,1},\chi_{1,1}$ 
& $D_{16} \times C_3$, $C_2$, $C_2$ & $-$ & $-$ \\ 
\hline

$7$ & $\RV_2:2$ & $\langle \left(\begin{smallmatrix} 3 & 0 \\ 
0 & 3 \end{smallmatrix} \right), 
\left(\begin{smallmatrix} 1 & 0 \\ 0 & -1 \end{smallmatrix} \right), 
\left(\begin{smallmatrix} 2 & 1 \\ 5 & 2 \end{smallmatrix} \right) \rangle$
& $\chi_{0,0},\chi_{0,1}$ & $SD_{32} \times C_3$, $C_2$ & $-$ & $-$\\ 
\hline

$13$ & $\MM$ & $\langle \left(\begin{smallmatrix} 1 & 0 \\ 
0 & 8 \end{smallmatrix} \right), 
\left(\begin{smallmatrix} 2 & 0 \\ 0 & 2 \end{smallmatrix} \right), 
\left(\begin{smallmatrix} 10 & 9 \\ 5 & 2 \end{smallmatrix} \right) \rangle$ 
& $\chi_{0,0},\chi_{0,1},\chi_{1,1}$ & $C_3 \times 4.S_4$, $C_4$, $C_3$  
& $[\a\b]$ & $C_3$ \\ 
\hline
\end{tabular}
\end{table}
\end{landscape}

\begin{table}
\renewcommand{\arraystretch}{1.4}
\centering
\caption{$\w(\F,0)$ and $\m(\F,0,d)$ for $d=2,3$ }
\label{t:rv3}
\begin{tabular}{|c|c|c|c|c|c|}

\hline
 $p$ & $\F$ & $\Out^*_\F(S)$  & $\m(\F,0,2)$ & $\m(\F,0,3)$ & $\w(\F,0)$ \\ 
\hline

$3 \nmid (p-1)$ & $\PSL_3(p)$ & $C_{p-1}$  & $p-1$ & $p^2$ & $p^2-1$ \\ 
\hline

$3 \nmid (p-1)$ & $\PSL_3(p):2$ & $C_{p-1}.C_2$ & $(p+5)/2$ & $p(p+3)/2$ 
& $(p-1)(p+4)/2$ \\ 
\hline

$3 \mid (p-1)$ & $\PSL_3(p)$ & $C_{(p-1)/3}$ & $(p-1)/3$ & $(p^2+8)/3$ 
& $(p^2-1)/3$ \\ 
\hline

$3 \mid (p-1)$ & $\PSL_3(p):2$ & $C_{(p-1)/3}.C_2$ & $(p+17)/6$ 
& $(p+1)(p+8)/6$ & $(p-1)(p+10)/6$ \\ 
\hline

$3 \mid (p-1)$ & $\PSL_3(p):3$ & $C_{p-1}$  & $p-1$ & $p^2$ & $p^2-1$ \\ 
\hline

$3 \mid (p-1)$ & $\PSL_3(p):S_3$ & $C_{p-1}.C_2$  & $(p+5)/2$ 
& $p(p+3)/2$ & $(p-1)(p+4)/2$ \\ 
\hline

$3$ & $^2F_4(2)'$ & $C_4$ & $4$ & $9$ & $9$ \\ \hline
$3$ & $J_4$ & $Q_8$ & $5$ & $9$ & $9$ \\ \hline
$5$ & $\Th$ & $\SL_2(3)$ & $7$ & $20$ & $20$ \\ \hline
$7$ & $\He$ & $C_3$ & $3$ & $20$ & $10$ \\ \hline
$7$ & $\He:2$ & $C_6$ & $6$ & $25$ & $20$ \\ \hline

$7$ & $\Fi'_{24}$  & $C_6$ & $6$ & $25$ & $22$ \\ \hline
$7$ & $\Fi_{24}$ & $D_{12}$ & $6$ & $35$ & $29$ \\ \hline
$7$ & $\RV_1$ & $D_{12}$ & $6$ & $35$ & $35$ \\ \hline
$7$ & $\ON$ & $C_4$ & $4$ & $20$ & $19$ \\ \hline
$7$ & $\ON:2$ &  $C_8$ & $8$ & $25$ & $23$ \\ \hline
$7$ & $\RV_2$ &  $C_8$ & $8$ & $25$ & $25$ \\ \hline
$7$ & $\RV_2:2$ &  $Q_{16}$ & $7$ & $35$ & $35$ \\ \hline
$13$ & $\MM$ & $\SL_2(3)$ & $7$ & $55$ & $52$ \\ \hline
\end{tabular}
\end{table}
\tiny{
\begin{table}
\renewcommand{\arraystretch}{1.4}
\centering
\caption{Counts of character degrees of $G$ for $\PSL_3(p) \le G \le \Aut(\PSL_3(p))$ }
\label{t:degrees}
\begin{tabular}{|c|c|c|c|c|c|c|}
\hline
$p$ & $3 \nmid p-1$ & $3 \nmid p-1$ & $3 \mid p-1$ & $3 \mid p-1$ & $3 \mid p-1$ & $3 \mid p-1$ \\
$\chi(1)$ & $\PSL_3(p)$ & $\PSL_3(p):2$ & $\PSL_3(p)$ & $\PSL_3(p):2$ & $\PSL_3(p):3$ & $\PSL_3(p):S_3$ \\  \hline $1$ & $1$ & $2$ & $1$ & $2$ & $3$ & $2$ \\ \hline
$2$ & $-$ & $-$ & $-$ & $-$ & $-$ & $1$ \\ \hline
$p(p+1)$ & $1$ & $2$ & $1$ & $2$ & $3$ & $2$ \\ \hline
$2p(p+1)$ & $-$ & $-$ & $-$ & $-$ & $-$ & $1$ \\ \hline
$p^3$ & $1$ & $2$ & $1$ & $2$ & $3$ & $2$ \\ \hline
$2p^3$ & $-$ & $-$ & $-$ & $-$ & $-$ & $1$ \\ \hline
$p^2+p+1$ & $p-2$ & $2$ & $\frac{p-4}{3}$ & $2$ & $p-4$ & $2$ \\ \hline
$2(p^2+p+1)$ & $-$ & $\frac{p-3}{2}$ & $-$ & $\frac{p-7}{6}$ & $-$ & $\frac{p-5}{2}$ \\ \hline
$p(p^2+p+1)$ & $p-2$ & $2$ & $\frac{p-4}{3}$ & $2$ & $p-4$ & $2$ \\ \hline
$2p(p^2+p+1)$ & $-$ & $\frac{p-3}{2}$ & $-$ & $\frac{p-7}{6}$ & $-$  & $\frac{p-5}{2}$ \\ \hline
$(p-1)(p^2+p+1)$ & $\frac{p(p-1)}{2}$ & $p-1$ & $\frac{p(p-1)}{6}$ & $p-1$ & $\frac{p(p-1)}{2}$ & $p-1$ \\ \hline
$2(p-1)(p^2+p+1)$ & $-$ & $\frac{(p-1)^2}{4}$ & $-$ & $\frac{(p-1)(p-3)}{12}$ & $-$ & $\frac{(p-1)^2}{4}$ \\ \hline
$(p+1)(p^2+p+1)$ & $\frac{(p-2)(p-3)}{6}$ & $p-3$ & $\frac{(p-1)(p-4)}{18}$ & $p-5$ & $\frac{p^2-5p+10}{6}$ & $p-3$ \\ \hline
$2(p+1)(p^2+p+1)$ & $-$ & $\frac{(p-3)(p-5)}{12}$ & $-$ & $\frac{(p-7)^2}{36}$ & $-$ & $\frac{p^2-8p+19}{12}$ \\ \hline
$\frac{(p+1)(p^2+p+1)}{3}$ & $-$ & $-$ & $3$ & $2$ & $-$ & $-$ \\ \hline
$2\frac{(p+1)(p^2+p+1)}{3}$ & $-$ & $-$ & $-$ & $1$ & $-$ & $-$ \\ \hline
$(p+1)(p-1)^2$ & $\frac{p(p+1)}{3}$ & $-$ & $\frac{(p+2)(p-1)}{9}$ & $-$  & $\frac{(p-1)(p+2)}{3}$ & $-$ \\ \hline
$2(p+1)(p-1)^2$ & $-$ & $\frac{p(p+1)}{6}$ & $-$ & $\frac{(p-1)(p+2)}{18}$ & $-$ & $\frac{(p-1)(p+2)}{6}$ \\
\hline
\end{tabular}
\end{table}
}

\bibliographystyle{amsalpha}
\bibliography{rv}

\providecommand{\bysame}{\leavevmode\hbox to3em{\hrulefill}\thinspace}
\providecommand{\MR}{\relax\ifhmode\unskip\space\fi MR }
\providecommand{\MRhref}[2]{%
  \href{http://www.ams.org/mathscinet-getitem?mr=#1}{#2}
}
\providecommand{\href}[2]{#2}
\begin{thebibliography}{GMRS04}

\bibitem[AKO11]{AschbacherKessarOliver2011}
Michael Aschbacher, Radha Kessar, and Bob Oliver, \emph{Fusion systems in
  algebra and topology}, London Mathematical Society Lecture Note Series, vol.
  391, Cambridge University Press, Cambridge, 2011. \MR{2848834}

\bibitem[An98]{an1998alperin}
Jianbei An, \emph{The {A}lperin and {D}ade conjectures for {R}ee groups
  {$^2F_4(q^2)$} in non-defining characteristics}, J. Algebra \textbf{203}
  (1998), no.~1, 30--49. \MR{1620701}

\bibitem[AOW03]{an2003alperin}
Jianbei An, E.~A. O'Brien, and R.~A. Wilson, \emph{The {A}lperin weight
  conjecture and {D}ade's conjecture for the simple group {$J_4$}}, LMS J.
  Comput. Math. \textbf{6} (2003), 119--140. \MR{1998146}

\bibitem[BCP97]{Magma}
Wieb Bosma, John Cannon, and Catherine Playoust, \emph{The {M}agma algebra
  system. {I}. {T}he user language}, J. Symbolic Comput. \textbf{24} (1997),
  no.~3-4, 235--265, Computational algebra and number theory (London, 1993).
  \MR{1484478}

\bibitem[Bla72]{Blackburn1972}
Norman Blackburn, \emph{Some homology groups of wreath products}, Illinois J.
  Math. \textbf{16} (1972), 116--129. \MR{0291294}

\bibitem[CR90]{CurtisReiner1990}
Charles~W. Curtis and Irving Reiner, \emph{Methods of representation theory.
  {V}ol. {I}}, Wiley Classics Library, John Wiley \& Sons Inc., New York, 1990,
  With applications to finite groups and orders, Reprint of the 1981 original,
  A Wiley-Interscience Publication.

\bibitem[Eat04]{eaton2004equivalence}
Charles~W. Eaton, \emph{The equivalence of some conjectures of {D}ade and
  {R}obinson}, J. Algebra \textbf{271} (2004), no.~2, 638--651. \MR{2025544}

\bibitem[EM14]{EatonMoreto2014}
Charles~W. Eaton and Alexander Moret\'{o}, \emph{Extending {B}rauer's height
  zero conjecture to blocks with nonabelian defect groups}, Int. Math. Res.
  Not. IMRN (2014), no.~20, 5581--5601. \MR{3271182}

\bibitem[FH91]{fulton2013representation}
William Fulton and Joe Harris, \emph{Representation theory}, Graduate Texts in
  Mathematics, vol. 129, Springer-Verlag, New York, 1991, A first course,
  Readings in Mathematics. \MR{1153249}

\bibitem[Gla71]{Glauberman1971}
G.~Glauberman, \emph{Global and local properties of finite groups}, Finite
  simple groups ({P}roc. {I}nstructional {C}onf., {O}xford, 1969), Academic
  Press, London, 1971, pp.~1--64. \MR{0352241 (50 \#4728)}

\bibitem[GLS98]{GLS3}
Daniel Gorenstein, Richard Lyons, and Ronald Solomon, \emph{The classification
  of the finite simple groups. {N}umber 3. {P}art {I}. {C}hapter {A}},
  Mathematical Surveys and Monographs, vol.~40, American Mathematical Society,
  Providence, RI, 1998, Almost simple $K$-groups. \MR{1490581 (98j:20011)}

\bibitem[GMRS04]{gluck2004solution}
David Gluck, Kay Magaard, Udo Riese, and Peter Schmid, \emph{The solution of
  the {$k(GV)$}-problem}, J. Algebra \textbf{279} (2004), no.~2, 694--719.
  \MR{2078936}

\bibitem[Gor80]{Gorenstein1980}
Daniel Gorenstein, \emph{Finite groups}, second ed., Chelsea Publishing Co.,
  New York, 1980.

\bibitem[GS11]{gill2011real}
Nick Gill and Anupam Singh, \emph{Real and strongly real classes in {${\rm
  SL}_n(q)$}}, J. Group Theory \textbf{14} (2011), no.~3, 437--459.
  \MR{2794377}

\bibitem[Hum06]{Humphreys2006}
James~E. Humphreys, \emph{Modular representations of finite groups of {L}ie
  type}, London Mathematical Society Lecture Note Series, vol. 326, Cambridge
  University Press, Cambridge, 2006. \MR{2199819}

\bibitem[Kar87]{Karpilovsky1987}
Gregory Karpilovsky, \emph{The {S}chur multiplier}, London Mathematical Society
  Monographs. New Series, vol.~2, The Clarendon Press, Oxford University Press,
  New York, 1987. \MR{1200015}

\bibitem[KLLS19]{KLLS2019}
Radha Kessar, Markus Linckelmann, Justin Lynd, and Jason Semeraro, \emph{Weight
  conjectures for fusion systems}, Adv. Math. \textbf{357} (2019), 106825, 40.
  \MR{4013803}

\bibitem[KS08]{KessarStancu2008}
Radha Kessar and Radu Stancu, \emph{A reduction theorem for fusion systems of
  blocks}, J. Algebra \textbf{319} (2008), no.~2, 806--823.

\bibitem[KSST24]{KSST}
Radha Kessar, Jason Semeraro, Patrick Serwene, and {\.{I}}pek Tuvay,
  \emph{Weight conjectures for {P}arker--{S}emeraro fusion systems}, 2024,
  arXiv:2407.07211 [math.RT].

\bibitem[K{\"{u}}l87]{kulshammer1987remark}
Burkhard K{\"{u}}lshammer, \emph{A remark on conjectures in modular
  representation theory}, Arch. Math. (Basel) \textbf{49} (1987), no.~5,
  396--399. \MR{915912}

\bibitem[Lin09]{Linckelmann2009}
Markus Linckelmann, \emph{The orbit space of a fusion system is contractible},
  Proc. Lond. Math. Soc. (3) \textbf{98} (2009), no.~1, 191--216. \MR{2472165}

\bibitem[Lin19a]{LinckelmannBT2}
\bysame, \emph{The block theory of finite group algebras}, London Mathematical
  Society Student Texts, vol.~2, Cambridge University Press, 2019.

\bibitem[Lin19b]{LinckelmannBT1}
\bysame, \emph{The block theory of finite group algebras}, London Mathematical
  Society Student Texts, vol.~1, Cambridge University Press, 2019.

\bibitem[LS23]{LyndSemeraro2023}
Justin Lynd and Jason Semeraro, \emph{Weights in a {B}enson-{S}olomon block},
  Forum Math. Sigma \textbf{11} (2023), Paper No. e60, 30. \MR{4612095}

\bibitem[MN06]{MalleNavarro2006}
Gunter Malle and Gabriel Navarro, \emph{Inequalities for some blocks of finite
  groups}, Arch. Math. (Basel) \textbf{87} (2006), no.~5, 390--399.

\bibitem[MR17]{MalleRobinson2017}
Gunter Malle and Geoffrey~R. Robinson, \emph{On the number of simple modules in
  a block of a finite group}, J. Algebra \textbf{475} (2017), 423--438.
  \MR{3612478}

\bibitem[Nag62]{nagao1962conjecture}
Hirosi Nagao, \emph{On a conjecture of {B}rauer for {$p$}-solvable groups}, J.
  Math. Osaka City Univ. \textbf{13} (1962), 35--38. \MR{152569}

\bibitem[Nar07]{Narasaki2007}
Ryo Narasaki, \emph{Character values and {D}ade's conjecture}, Math. J. Okayama
  Univ. \textbf{49} (2007), 1--36.

\bibitem[NU09]{NarasakiUno2009}
Ryo Narasaki and Katsuhiro Uno, \emph{Isometries and extra special {S}ylow
  groups of order {$p^3$}}, J. Algebra \textbf{322} (2009), no.~6, 2027--2068.

\bibitem[Rob96]{Robinson1996}
Geoffrey~R. Robinson, \emph{Local structure, vertices and {A}lperin's
  conjecture}, Proc. London Math. Soc. (3) \textbf{72} (1996), no.~2, 312--330.
  \MR{1367081}

\bibitem[Rob04]{Robinsonweight2004}
\bysame, \emph{Weight conjectures for ordinary characters}, J. Algebra
  \textbf{276} (2004), no.~2, 761--775. \MR{2058466}

\bibitem[RV04]{RuizViruel2004}
Albert Ruiz and Antonio Viruel, \emph{The classification of {$p$}-local finite
  groups over the extraspecial group of order {$p\sp 3$} and exponent {$p$}},
  Math. Z. \textbf{248} (2004), no.~1, 45--65.

\bibitem[Sem23]{Semeraro2021}
Jason Semeraro, \emph{A 2-compact group as a spets}, Exp. Math. \textbf{32}
  (2023), no.~1, 140--155. \MR{4574426}

\bibitem[SF73]{simpson1973character}
William~A. Simpson and J.~Sutherland Frame, \emph{The character tables for
  {${\rm SL}(3,\,q)$}, {${\rm SU}(3,\,q\sp{2})$}, {${\rm PSL}(3,\,q)$}, {${\rm
  PSU}(3,\,q\sp{2})$}}, Canadian J. Math. \textbf{25} (1973), 486--494.
  \MR{335618}

\bibitem[Ste51]{steinberg1951representations}
Robert Steinberg, \emph{The representations of {GL}$(3, q)$, {GL}$(4, q)$,
  {PGL}$(3, q)$, and {PGL}$(4, q)$}, Canadian Journal of Mathematics \textbf{3}
  (1951), 225--235.

\bibitem[Usa01]{usami2001principal}
Yoko Usami, \emph{Principal blocks with extra-special defect groups of order
  27}, Groups and combinatorics---in memory of {M}ichio {S}uzuki, Adv. Stud.
  Pure Math., vol.~32, Math. Soc. Japan, Tokyo, 2001, pp.~413--421.
  \MR{1893508}

\bibitem[Wie71]{Wiegold71}
James Wiegold, \emph{The multiplicator of a direct product}, Quart. J. Math.
  Oxford Ser. (2) \textbf{22} (1971), 103--105.

\end{thebibliography}
\end{document}